\documentclass[12pt,reqno]{amsart}
\usepackage{amsmath, amssymb, amsthm}
\usepackage{geometry}
\usepackage{hyperref}
\usepackage{cite}
\usepackage{slashed} 
\usepackage{tikz-cd}
\usepackage{tikz}
\usepackage{adjustbox}
\usepackage{mathtools}
\usetikzlibrary{arrows.meta,decorations.markings,calc}
\tikzset{
  edge/.style={very thick},
  orient/.style={
    postaction={decorate},
    decoration={markings, mark=at position 0.55 with {\arrow{Stealth[length=7pt]}}}
  }
}
\usepackage{leftidx} 
\usepackage{quiver}
\usepackage{enumitem}

\theoremstyle{plain}
\newtheorem{theorem}{Theorem}[section] 
\newtheorem{lemma}[theorem]{Lemma}
\newtheorem{proposition}[theorem]{Proposition}

\theoremstyle{definition}
\newtheorem{theoremlet}{Theorem}

\newtheorem{remark}[theorem]{Remark}

\title{A Non-trivial index difference on Surfaces  \\ of genus at least $3$}
\author{Samuel Lockman$^{\ast}$}
\address{Universität Regensburg, Fakultät für Mathematik, 93040 Regensburg, Germany}
\email{\href{mailto:samuel.lockman@mathematik.uni-regensburg.de}{samuel.lockman@mathematik.uni-regensburg.de}}
\thanks{{$^{\ast}$ Funded by the Deutsche Forschungsgemeinschaft (DFG, German Research Foundation) – Project numbers
224262486; 
313840899}} 

\DeclareMathOperator{\id}{Id}

\DeclareMathOperator{\inddiff}{ind-diff}
\DeclareMathOperator{\sign}{sign}

\newcommand{\KO}{\mathbf{KO}}

\newcommand{\spin}{\text{Spin}}

\newcommand{\esc}[1]{\langle{#1}\rangle}

\newcommand{\cl}{\mathbf{Cl}}

\newcommand{\Diff}[1]{\mathrm{Diff}(#1)}
\newcommand{\SDiff}[1]{\mathrm{SDiff}(#1)}

\begin{document}
\begin{abstract}
    For every closed surface of genus at least $3$, equipped with any bounding spin structure, we show that the index difference, viewed as a map from the fundamental group of the space of Dirac-invertible Riemannian metrics to $\KO^{-4}(*)$, is non-trivial. If the genus is at least $5$, we also show that the aforementioned index difference is surjective. For products of two closed surfaces of genus at least $3$, equipped with any spin structure, we prove that the corresponding space of Dirac-invertible Riemannian metrics is not contractible. We discuss the relationship of this result to the existence of metrics with harmonic spinors in dimension~$4$.
\end{abstract}

\maketitle
\markboth{\protect\normalfont\scshape Samuel Lockman}{\protect\normalfont\scshape Samuel Lockman}

\setcounter{tocdepth}{1}
\tableofcontents

\section{Introduction}
Let $M$ be a closed and connected spin manifold of dimension $n$ and let $\mathcal{R}(M)$ denote the space of Riemannian metrics on $M$, equipped with the $C^{\infty}$ topology. Index theory has been one of the main tools for studying the space $\mathcal{R}_{\mathrm{psc}}(M) \subset \mathcal{R}(M)$ of positive scalar curvature metrics. Building on the surgery results by Gromov-Lawson \cite{GL-classification} and Schoen-Yau \cite{Schoen-yau} (see also \cite{Ebert-Frenck}), Stolz \cite{Stolz-GL-conj} proved that any closed and simply connected spin manifold of dimension at least $5$ admits a metric of positive scalar curvature if and only if the $\KO^{-n}(*)$-valued index of the Dirac operator vanishes. To study the homotopy type of the space $\mathcal{R}_{\mathrm{psc}}(M)$, Hitchin \cite[Section~4.4]{Hitchin}, assuming $\mathcal{R}_{\mathrm{psc}}(M)$ is non-empty, constructed a map 
\begin{align}\label{ind-diff_psc}
    \inddiff : \pi_{\ell}(\mathcal{R}_{\mathrm{psc}}(M), g_0) \to \KO^{-(n + \ell + 1)}(*),
\end{align} 
which was later called the index difference (see also \cite{Ebert2016}). Using the Atiyah-Singer family index theorem \cite{Index-of-elliptic-ops-IV}, Hitchin showed that this map is non-trivial if $n \equiv 0,1 \mod{8}$ and $\ell = 0$  or $n \equiv 0,7 \mod{8}$ and $\ell = 1$. This map has been studied further by many \cite{Botvinnik-Ebert-Randal-Williams, Crowley-Schick-Steimle, Crowley-Schick, Hanke-Schick-Steimle, GL, Kreck-Stolz, Ebert2016}. For example, Botvinnik, Ebert, and Randal-Williams \cite{Botvinnik-Ebert-Randal-Williams} showed that if $n \geq 6$, then the map \eqref{ind-diff_psc} is non-trivial for all $\ell \geq 0$, whenever the codomain is non-trivial. 

The manifolds considered in this paper do not admit metrics of positive scalar curvature. Instead, we are interested in the larger space of metrics for which the Dirac operator is invertible, denoted by $\mathcal{R}_{\mathrm{inv}}(M)$. This is precisely the space needed to construct the index difference. Specifically, the map \eqref{ind-diff_psc} factors through the map
\begin{align}\label{ind-diff}
    \inddiff : \pi_{\ell}(\mathcal{R}_{\mathrm{inv}}(M), g_0) \to \KO^{-(n + \ell + 1)}(*),
\end{align}
which only needs $\mathcal{R}_{\mathrm{inv}}(M) \neq \emptyset$ to be defined. In addition to being a tool to study the homotopy type of the space $\mathcal{R}_{\mathrm{inv}}(M)$, it can also be used to prove the existence of Riemannian metrics with a harmonic spinor. Indeed, if one can show that this map is non-trivial for some $\ell$, then the space $\mathcal{R}_{\mathrm{inv}}(M)$ is not equal to the contractible space $\mathcal{R}(M)$ of all Riemannian metrics and hence there is a metric on $M$ with a harmonic spinor. Using this technique, Hitchin \cite[Theorem~4.5]{Hitchin} proved that if $n\geq 3$, and $n \equiv 0,1,7 \mod{8}$, then $M$ admits a metric with a harmonic spinor. Bär \cite[Theorem A]{bar-harm} proved the analogous statement for $n \equiv 3 \mod{4}$. Although not explicitly stated, Bär shows that the index difference for $\ell = 0$ is non-trivial in these dimensions, see also \cite{Dahl, Waterstraat-metrics}. Inspired by these results, Bär and Dahl \cite[Conjecture A]{bar-dahl} conjectured that any closed spin manifold of dimension at least $3$ admits a metric with a harmonic spinor. Seeger \cite{Seeger-thesis} proved the conjecture for spheres of dimension $n \equiv 0 \mod{4}$. Dahl \cite[Theorem 4.1 \& Corollary 4.2]{Dahl} proved the conjecture for certain simply connected manifolds of dimension at least $5$ with a cyclic group action; in particular, Dahl proved the conjecture for any sphere of dimension at least~$5$. Crowley, Schick, and Steimle \cite[Theorem 1.4]{Crowley-Schick-Steimle} proved the conjecture in all cases with $n \geq 6$. Specifically, they show the following.
\begin{theorem}[Crowley, Schick, Steimle, 2018]
    If $\dim M = n \geq 6$, then for all $\ell \geq 0$ with $\ell + n \equiv 0,1 \mod{8}$, and for all $g_0 \in \mathcal{R}_{\mathrm{inv}}(M)$ the map
    \begin{align*}
        \mathrm{\inddiff} :\pi_{\ell}(\mathcal{R}_{\mathrm{inv}}(M), g_0) \to \KO^{-(n+\ell+1)}(*)
    \end{align*}
    is a split surjection.
\end{theorem}
We see that the conjecture by Bär and Dahl is still open for $n \in \{4,5\}$. The case $n=2$ is not mentioned in this conjecture because of the following. Hitchin \cite[Proposition 2.3]{Hitchin} proved that for any closed surface $\Sigma_{\gamma}$, equipped with any spin structure, we have for any metric $g \in \mathcal{R}(\Sigma_{\gamma})$ that
\begin{align}
    \dim_{\mathbb{H}} \ker{\slashed{D}^g} \leq \Big\lfloor \frac{\gamma+1}{2}\Big\rfloor,
\end{align}
where $\slashed{D}^g$ denotes the Dirac operator acting on sections of the spinor bundle associated to the irreducible representation of $C\ell_2 \cong \mathbb{H}$. This means that if $\gamma \leq 2$ and $\Sigma_{\gamma}$ is equipped with any bounding spin structure, then there does not exist a metric with a harmonic spinor, i.e. $\mathcal{R}_{\mathrm{inv}}(\Sigma_{\gamma}) = \mathcal{R}(\Sigma_{\gamma})$. For any closed surface $\Sigma_{\gamma}$ of genus $\gamma \geq 2$ equipped with any spin structure~$\mathfrak{s}$, Bär and Schmutz \cite{bar-schmutz} showed that for any integer $k$ with $0 \leq 2k \leq \gamma +1$ and with $k$ even if the spin structure $\mathfrak{s}$ is bounding and $k$ odd if the spin structure $\mathfrak{s}$ is non-bounding, then there is a Riemannian metric $g$ such that $\dim_{\mathbb{H}} \ker{\slashed{D}^g} = k$. Hence any closed spin surface $\Sigma_{\gamma}$ of genus $\gamma \geq 3$ admits a metric with a harmonic spinor. An alternative proof of the existence of a metric with a harmonic spinor was given by Ammann, Wei{\ss}, and Witt \cite{Ammann-Weiss-Witt}.

The index difference is also related to the following. Let $\mathcal{R}_{\mathrm{min}}(M) \subset \mathcal{R}(M)$ be the set of Riemannian metrics for which the kernel of the corresponding Dirac operator attains the lower bound given by the $\KO^{-n}(*)$-valued index. Note that if this index vanishes, we have that  $\mathcal{R}_{\mathrm{inv}}(M) = \mathcal{R}_{\mathrm{min}}(M)$. Maier \cite{Maier} showed that $\mathcal{R}_{\mathrm{min}}(M)$ is dense in $\mathcal{R}(M)$ if $n \leq 4$. Ammann, Dahl, and Humbert \cite{Ammann-Dahl-Humbert} showed the corresponding density result in all dimensions. Further topological information about the space $\mathcal{R}_{\mathrm{min}}(M)$ was obtained by Ammann and Dahl \cite{Ammann-Dahl}, where they showed the following theorem. 
\begin{theorem}[Ammann, Dahl, 2025]\label{Ammann-Dahl}
    If $\dim M \in \{2,4\}$, then $\mathcal{R}_{\mathrm{min}}(M)$ is connected.
\end{theorem}
The main goal of this article is to show that this is sharp in dimension $2$, in the following sense.
\begin{theoremlet}\label{main theorem}
    Let $\Sigma_{\gamma}$ be a closed surface of genus $\gamma \geq 3$, equipped with any bounding spin structure. Then, for every $g_0 \in \mathcal{R}_{\mathrm{inv}}(\Sigma_{\gamma})$, the index difference
    \begin{align}
         \inddiff : \pi_1(\mathcal{R}_{\mathrm{inv}}(\Sigma_{\gamma}), g_0) \to \KO^{-4}(*),
    \end{align}
    is a non-trivial homomorphism. Furthermore, if $\gamma \geq 5$, then the above index difference is surjective.
\end{theoremlet}
\begin{remark}
    Note that Theorem \ref{main theorem} gives a topological proof of the aforementioned result that any spin surface of genus at least $3$ admits a metric with a harmonic spinor. Also note that $\gamma \geq 3$ in this theorem is sharp. Indeed, as mentioned above, we have for any $\beta \leq 2$ that $\mathcal{R}_{\mathrm{inv}}(\Sigma_{\beta}) = \mathcal{R}(\Sigma_{\beta})$, when $\Sigma_{\beta}$ is equipped with any bounding spin structure.
\end{remark}
We also conclude the following result.
\begin{theoremlet}\label{second main theorem}
    Let $\Sigma_{\gamma^1}\times \Sigma_{\gamma^2}$ be a product of closed surfaces of genera ${\gamma}^1, {\gamma}^2 \geq 3$, equipped with any spin structure. Then, at least one of the groups 
    \begin{align*}
        &\pi_1(\mathcal{R}_{\mathrm{inv}}(\Sigma_{\gamma^1}\times \Sigma_{\gamma^2})), \\
        &\pi_2(\mathcal{R}_{\mathrm{inv}}(\Sigma_{\gamma^1}\times \Sigma_{\gamma^2})), \\
        &\pi_3(\mathcal{R}_{\mathrm{inv}}(\Sigma_{\gamma^1}\times \Sigma_{\gamma^2}))
    \end{align*}
    is non-trivial.
\end{theoremlet}
\begin{remark}\label{homotopy-equivalence-remark}
    It is believed that if two non-empty closed spin manifolds $M_1$ and $M_2$ of dimension at least $3$ are spin bordant to each other, then there is a homotopy equivalence $\mathcal{R}_{\mathrm{inv}}(M_1) \overset{H}{\to}~\mathcal{R}_{\mathrm{inv}}(M_2)$. One should view this as the spinoral analogue of the Chernysh/Walsh surgery result \cite{Chernysh, Walsh, Ebert-Frenck} for positive scalar curvature. Substantial progress towards such a homotopy equivalence $H$ has been made in \cite{grossePederzani, Pederzani}. If this is proven, then Theorem \ref{second main theorem} proves the case $n = 4$ of the conjecture by Bär and Dahl discussed above. To see this, recall that the $\widehat{A}$-genus gives an isomorphism $\Omega_4^{\mathrm{spin}} \overset{\cong}{\to} 2\mathbb{Z}$, where $\Omega_4^{\mathrm{spin}}$ is the spin bordism group in dimension~$4$. Hence, if $M$ is a closed $4$-dimensional spin manifold with vanishing $\widehat{A}$-genus, then it is spin bordant to $\Sigma_3 \times \Sigma_3$, equipped with any spin structure. By Theorem \ref{second main theorem}, we know that $\mathcal{R}_{\mathrm{inv}}(\Sigma_3 \times \Sigma_3)$ is not contractible and therefore, assuming that such a homotopy equivalence $H: \mathcal{R}_{\mathrm{inv}}(\Sigma_3 \times \Sigma_3) \to \mathcal{R}_{\mathrm{inv}}(M)$ exists, we conclude that there is a metric $g$ on $M$ with a harmonic spinor.
\end{remark}
\begin{remark}
    Note also that Theorem \ref{main theorem} shows that such a homotopy equivalence $H$ as mentioned in Remark \ref{homotopy-equivalence-remark} cannot exist in dimension $2$.
\end{remark}
\begin{remark}\label{Ammann-Dahl conjecture}
    There is an unwritten conjecture by Ammann and Dahl which says that in dimension $4$, the space $\mathcal{R}_{\mathrm{inv}}(M)$ is $2$-connected if $M$ has vanishing $\widehat{A}$-genus. If this conjecture is proven, then by the assumptions of Theorem \ref{second main theorem}, we get that the index difference
    \begin{align}
        \mathrm{\inddiff} :\pi_3(\mathcal{R}_{\mathrm{inv}}(\Sigma_{\gamma^1} \times \Sigma_{\gamma^2}), g) \to \KO^{-8}(*)
    \end{align}
    is non-trivial. The last passage of Section \ref{last_section} explains how this is deduced.
\end{remark}

\textbf{Acknowledgements:} The author thanks Bernd Ammann, Jonathan Bowden, Ulrich Bunke, and Oscar Randal-Williams for their helpful discussions on topics related to this article.

\section{Preliminaries and notation}
All spaces used in this section will be assumed sufficiently nice so that $G$-principal bundles are classified by maps into $\mathrm{B}G$, and that the standard theory of covering spaces holds. For example, the reader may assume that all spaces are paracompact, Hausdorff, locally path-connected and semilocally simply connected. For the proofs of Theorem \ref{main theorem} and Theorem \ref{second main theorem}, we will only use manifolds and CW-complexes. Mapping spaces of the form $C^{\infty}(X,Y)$ used in this article will carry the weak $C^{\infty}$ topology.

\subsection{Spin structures on fiber bundles}\label{spin structures on fiber bundles}
In this section, we recall the necessary definitions to state the obstruction for a smooth fiber bundle to admit a spin structure, extending a given spin structure on the fiber. We refer the reader to \cite[Chapter 3]{Ebert-Dissertation} for a detailed presentation. 

    For $n\geq 2$, let $ \theta : \widetilde{\mathrm{GL}^+_n} \to \mathrm{GL}^+_n$ be the connected double covering group, where $\mathrm{GL}^+_n$ denotes the identity component of the general linear group $\mathrm{GL}(n, \mathbb{R})$. For $n=1$, we let $ \theta : \widetilde{\mathrm{GL}^+_n} \to \mathrm{GL}^+_n$ be the disconnected double covering. Let $E \to X$ be an oriented real vector bundle of rank $n$ and let $\mathrm{GL}^+(E)$ be the bundle of positively oriented frames on $E$. A spin structure $\mathfrak{s} = (\widetilde{\mathrm{GL}^+}(E), \Theta)$ on $E$ is a lift of the structure group of $\mathrm{GL}^+(E)$ along $\theta$. This means that $\Theta : \widetilde{\mathrm{GL}^+}(E) \to \mathrm{GL}^+(E)$ is a $\theta$-equivariant map of principal bundles, covering the identity map on $X$. An isomorphism between two spin structures $\mathfrak{s}_0 = (\widetilde{\mathrm{GL}^+}(E)_0, \Theta_0)$, and $\mathfrak{s}_1 = (\widetilde{\mathrm{GL}^+}(E)_1, \Theta_1)$ on $E$ is an isomorphism $F : \widetilde{\mathrm{GL}^+}(E)_0 \to \widetilde{\mathrm{GL}^+}(E)_1$ of $\widetilde{\mathrm{GL}^+_n}$-principal bundles, covering the identity on $X$, such that $\Theta_0 = \Theta_1 \circ F$. An orientable vector bundle which can be equipped with a spin structure is called a spinnable vector bundle and an oriented vector bundle which is equipped with a spin structure is called a spin vector bundle. An oriented vector bundle $E \to X$ is spinnable if and only if the second Stiefel–Whitney class $w_2(E) \in H^2(X, \mathbb{Z}/2\mathbb{Z})$ vanishes. If $\mathfrak{s} = (\widetilde{\mathrm{GL}^+}(E), \Theta)$ is a spin structure on $E \to X$ and $f : Y \to X$ is a continuous map, then the pullback $f^{*}\mathfrak{s} := (f^{*}\widetilde{\mathrm{GL}^+}(E), f^*\Theta)$ defines a spin structure on $f^*E \to Y$. A spin structure on an oriented smooth manifold $M$ is a spin structure on the tangent bundle $TM \to M$. A spinnable manifold is an orientable smooth manifold which can be equipped with a spin structure and a spin manifold is an oriented smooth manifold which has been equipped with a spin structure. 
    
    Let $\text{Diff}(M)$ denote the group of orientation-preserving diffeomorphisms of $M$. For $f \in \text{Diff}(M)$, we can pull back any spin structure $\mathfrak{s}$ on $M$ to a spin structure $f^*\mathfrak{s}$ on $f^*TM$ as above. Since this is not a spin structure on $M$, we use the isomorphism $df : TM \overset{\cong}{\to} f^*TM$ of oriented vector bundles covering the identity, to define the pullback of a spin structure $\mathfrak{s} = (\widetilde{\mathrm{GL}^+}(M), \Theta)$ on $M$, via $f \in \text{Diff}(M)$, as $f^*\mathfrak{s} := (f^*\widetilde{\mathrm{GL}^+}(M), df^{-1} \circ f^*\Theta)$. The group $\text{Diff}(M)$ acts on $\spin(M)$ from the right via pullback, where $\spin(M)$ is the set of isomorphism classes of spin structures on $M$. The stabilizer in $\Diff{M}$ of the class represented by $\mathfrak{s}$ is denoted by $\Diff{M,\mathfrak{s}}$. If $f_0, f_1 \in \Diff{M}$ are isotopic, then their aforementioned action on $\spin(M)$ coincides. Therefore, we get an action
\begin{align}\label{action of pi0}
    \spin(M) \times& \pi_0(\Diff M) \to \spin(M) \\
     ([\mathfrak{s}],[&f]) \mapsto [f^*\mathfrak{s}], \nonumber
\end{align}
and we note that $\pi_0(\Diff{M, \mathfrak{s}})$ is the stabilizer of $[\mathfrak{s}]$ in $\pi_0(\Diff M)$.

Assume now that $M$ is a connected, oriented, and smooth manifold equipped with some spin structure $\mathfrak{s}$. We say that a spin diffeomorphism of $(M, \mathfrak{s})$ is a diffeomorphism $f : M \to M$, together with a choice of isomorphism $F : f^*\mathfrak{s} \to \mathfrak{s}$ of spin structures on $M$. Let $\SDiff{M, \mathfrak{s}}$ denote the group of spin diffeomorphisms $(f, F)$. Note that the kernel of the forgetful surjection $\SDiff{M, \mathfrak{s}} \to \Diff{M, \mathfrak{s}}$ has two elements and hence we get the group extension
\begin{align}\label{group-extension}
    1 \to \mathbb{Z}/2\mathbb{Z} \to \SDiff{M, \mathfrak{s}} \to \Diff{M, \mathfrak{s}} \to 1.
\end{align}
Passing to classifying spaces, we get the homotopy fiber sequence
\begin{align}\label{fibration of group extension}
    \mathbb{RP}^{\infty} \to \mathrm{B}\SDiff{M, \mathfrak{s}} \to \mathrm{B}\Diff{M, \mathfrak{s}},
\end{align}
and we let $c(M, \mathfrak{s}) \in H^2(\mathrm{B}\hspace{0.4mm}\Diff{M, \mathfrak{s}}, \mathbb{Z}/2\mathbb{Z})$ denote the obstruction class for the existence of a section of (\ref{fibration of group extension}).

We say that a fiber bundle $p : Z \to B$ with fiber $M$ is smooth and oriented if the structure group is $\Diff{M}$. Let $p : Z \to B$ be a smooth, oriented fiber bundle with fiber $M$ being a connected, smooth manifold. For simplicity, we will also assume that $B$ is connected. Letting $Q_Z \to B$ be the corresponding $\Diff{M}$-principal bundle associated to $p: Z \to B$, we define the vertical tangent bundle of $Z \to B$ as $T^{v}Z := Q_Z \times_{\Diff{M}} TM \to Q_Z \times_{\Diff{M}} M \cong Z$, where $\Diff{M}$ acts on $TM$ via the tangent map. This is a vector bundle over $Z$ and the pullback along any fiber inclusion $M \cong Z_b \subset Z$ is isomorphic to the tangent bundle of $M$. Note that by composing with the map $p$, we can also view $T^vZ$ as a bundle of vector bundles $T^v Z \to B$ with fiber $TM \to M$. A spin structure on $Z \overset{p}{\to} B$ extending a spin structure $\mathfrak{s}$ on $M$ is a reduction of the structure group of $Q_Z$ to a $\Diff{M, \mathfrak{s}}$-principal bundle $Q'_Z$ along the inclusion $\Diff{M, \mathfrak{s}} \subset \Diff{M}$, followed by a lift of the structure group of $Q'_Z$ along the forgetful map $\SDiff{M,\mathfrak{s}} \to \Diff{M, \mathfrak{s}}$ to an $\SDiff{M,\mathfrak{s}}$-principal bundle $P_Z \to B$. A spin structure on $Z \to B$ induces a spin structure on the vector bundle $T^vZ \to Z$. If $p: Z \to B$ is equipped with a spin structure that extends the spin structure $\mathfrak{s}$ on $M$, then we say that $p : Z \to B$ is a spin $(M, \mathfrak{s})$-bundle.

The bundle $p: Z \to B$ is classified by a map $B \to \mathrm{B}\Diff{M}$, which we also denote by $p$. The induced map on fundamental groups gives a homomorphism
\begin{align*} 
    \rho : \pi_1(B, b_0) \overset{p_*}{\to} \pi_1(\mathrm{B}\Diff{M}, *) \overset{\cong}{\to} \pi_0(\Diff{M}),
\end{align*}
where the second map is the connecting homomorphism coming from the long exact sequence of homotopy groups corresponding to the fibration $E\Diff{M} \to \mathrm{B}\Diff{M}$. The map $\rho$ is called the monodromy representation of the bundle $p$ at the base point $b_0$. Since we assume $B$ to be connected, we will ignore any reference to a base point and simply refer to the monodromy representation. If the monodromy representation $\rho$ is contained in $\pi_0(\Diff{M,s})$, we write $\hat{p} : B \to B\Diff{M,s}$ to denote the classifying map corresponding to a $\Diff{M,s}$-reduction of the underlying $\Diff{M}$-principal bundle corresponding to $Z \to B$. The following proposition is due to Ebert \cite[Corollary 3.5.5]{Ebert-Dissertation}.

\begin{proposition}[Ebert, 2006]\label{spin-along-fibers}
    Let $p : Z \to B$ be a smooth fiber bundle with fiber $M$. Then, for any spin structure $\mathfrak{s}$ on $M$, there exists a spin structure on $p : Z \to B$, extending $\mathfrak{s}$, if and only if the image of the monodromy representation $\rho$ is contained inside $\pi_0(\Diff{M, \mathfrak{s}})$ and the class $\hat{p}^*c(M,\mathfrak{s}) \in H^2(B, \mathbb{Z}/2\mathbb{Z})$ vanishes.  
\end{proposition} 

\subsection{Riemannian metrics along the fibers of a fiber bundle}\label{metrics-along-fibers}
We continue with the notation from the previous section. Recall that $\mathcal{R}(M)$ denotes the space of Riemannian metrics on $M$. We let $\Diff{M}$ act on $\mathcal{R}(M)$ from the left via pushforward. This means that the action is given by
\begin{align}\label{diff-metric action}
    \Diff{M} \times \mathcal{R}(M) &\to \mathcal{R}(M) \\
    (f,g)& \mapsto (f^{-1})^*g =: f_*g \nonumber.
\end{align}
We form the associated fiber bundle $\mathcal{R}^{\mathrm{fib}}(T^vZ) = Q_Z \times_{\Diff{M}} \mathcal{R}(M) \to B$. If $Z \to B$ admits a spin structure extending the spin structure $\mathfrak{s}$ on the fiber $M$, then we will identify $P_Z \times_{\SDiff{M, \mathfrak{s}}} \mathcal{R}(M)$ and $Q_Z \times_{\Diff{M}} \mathcal{R}(M)$ via $[L \times \id]$, where $L : P_Z \to Q_Z$ denotes the data of the spin structure on $Z \to B$ as defined in the previous section. A continuous section of $\mathcal{R}^{\mathrm{fib}}(T^vZ) \to B$ is called a continuous family of Riemannian metrics along the fibers of $p : Z \to B$. Note that since $\mathcal{R}(M)$ is contractible, such a section always exists. Note also that a Riemannian metric $g : B \to \mathcal{R}^{\mathrm{fib}}(T^vZ)$ along the fibers of $Z \to B$ induces a bundle metric on the vector bundle $T^vZ \to Z$, and makes the fiber bundle $Z \to B$ into a bundle of Riemannian manifolds.

Suppose that $Z \to B$ is a spin $(M, \mathfrak{s})$-bundle and let $P_Z \to B$ denote its associated $\SDiff{M, \mathfrak{s}}$ bundle. For any spin diffeomorphism-invariant subset $\mathcal{S}(M) \subset \mathcal{R}(M)$, we can just as above form the associated fiber bundle $\mathcal{S}^{\mathrm{fib}}(T^vZ) := P_Z \times_{\SDiff{M, \mathfrak{s}}} \mathcal{S}(M) \to B$, which is a subbundle of $\mathcal{R}^{\mathrm{fib}}(T^vZ)$. We denote sections $B \to \mathcal{R}^{\mathrm{fib}}(T^vZ)$ which restrict to $C \to \mathcal{S}^{\mathrm{fib}}(T^vZ)\vert_C$ for some $C \subset B$ as $(B,C) \to (\mathcal{R}^{\mathrm{fib}}(T^vZ),\mathcal{S}^{\mathrm{fib}}(T^vZ))$ and call them sections of the map
\begin{align*}
    (\mathcal{R}^{\mathrm{fib}}(T^vZ),\mathcal{S}^{\mathrm{fib}}(T^vZ)) \to (B,C).
\end{align*} 
Pointed sections $g$ with $b_0 \in C$ and $g(b_0) = [h,g_0]$ are denoted as
\begin{align*}
    g : (B,C, b_0) \to (\mathcal{R}^{\mathrm{fib}}(T^vZ),\mathcal{S}^{\mathrm{fib}}(T^vZ), [h,g_0]).
\end{align*}

Let $b_0 \in B$ be a fixed $0$-cell and let $B^{\ell}$ denote the $\ell$-skeleton of $B$. If $\mathcal{S}(M)$ is $k$-connected, then for every $[h, g_0] \in \mathcal{S}^{\mathrm{fib}}(T^vZ)$, there is a section
\begin{align*}
    g : (B, B^{k+1}, b_0) \to (\mathcal{R}^{\mathrm{fib}}(T^vZ), \mathcal{S}^{\mathrm{fib}}(T^vZ), [h,g_0]).
\end{align*}
\begin{remark}
    The previous claim follows from standard obstruction theory, and is just a statement about fiber bundles over CW-complexes. An elementary way of proving this is via induction over the cells, see for example \cite[Page 21]{husemoller}.
\end{remark}
In this article, we will apply the above discussion to the space $\mathcal{R}_{\text{inv}}(M) \subset \mathcal{R}(M)$. We use the notation $\mathcal{R}^{\text{fib}}_{\text{inv}}(T^vZ): = P_Z \times_{\SDiff{M, \mathfrak{s}}} \mathcal{R}_{\text{inv}}(M)$.

\subsection{Metric spin structures}\label{metric spin structures}
The classical spin covering $\theta : \spin(n) \to \mathrm{SO}(n)$ is the co-restriction of the homomorphism $\theta$ given in Section \ref{spin structures on fiber bundles}, along the standard inclusion $\mathrm{SO}(n) \subset \mathrm{GL}^+_n$.  Let $E \to X$ be an oriented real vector bundle of rank $n$, with a positive definite bundle metric $g$ and let $\mathrm{SO}(E)$ be the $\mathrm{SO}(n)$-principal bundle of positively oriented orthonormal frames on $E$. A metric spin structure $\mathfrak{s}^g = ({\mathrm{Spin}}^g(E), \Theta)$ on $E$ is a lift of the structure group of $\mathrm{SO}(E)$ along $\theta$. A metric spin structure on an oriented Riemannian manifold is a metric spin structure on the tangent bundle of $M$. Let $\mathfrak{s}^g = ({\mathrm{Spin}}^g(E), \Theta)$ be a metric spin structure on an oriented vector bundle $E \to X$ of rank $n$. We define a corresponding Clifford bundle
\begin{align*}
    \slashed{\mathfrak{S}}^g(E) = \mathrm{Spin}^g(E) \times_{l} C\ell_n,
\end{align*}
where the representation $l$ is given by left multiplication. In case $E = TM$ is the tangent bundle of an oriented Riemannian spin manifold, then the bundle $\slashed{\mathfrak{S}}^g(M):= \slashed{\mathfrak{S}}^g(TM)$ is called the $\cl_n$-linear spinor bundle on the Riemannian spin manifold $(M, g, \mathfrak{s}^g)$. The bundle $\slashed{\mathfrak{S}}^g(M)$ carries the standard structures of Clifford multiplication $c : T^*M \otimes \slashed{\mathfrak{S}}^g(M) \to \slashed{\mathfrak{S}}^g(M)$ and a connection $\widetilde{\nabla}^g$ obtained by lifting the Levi-Civita connection on $TM$. This gives rise to the Dirac operator
\begin{align*}
    \slashed{D}^g : \Gamma(\slashed{\mathfrak{S}}^g(M)) \overset{\widetilde{\nabla}^g}{\to} \Gamma(T^*M \otimes \slashed{\mathfrak{S}}^g(M)) \overset{c}{\to} \Gamma(\slashed{\mathfrak{S}}^g(M)).
\end{align*}
The key feature of the $C\ell_n$-linear spinor bundle $\slashed{\mathfrak{S}}^g(M)$ is that it has a parallel right action of $C\ell_n$ which commutes with the action given by $c$, and therefore makes the Dirac operator $C\ell_n$-linear. Hence the kernel of the Dirac operator is a $C\ell_n$-module.

 Given an oriented vector bundle $E \to X$ with a spin structure $\mathfrak{s} = (\widetilde{\mathrm{GL}^+}(E), \Theta)$, we get for each choice of positive definite bundle metric $g$ on $E \to X$, a metric spin structure $\mathfrak{s}^g = ({\mathrm{Spin}}^g(E), \Theta)$, given as the restriction
\[
\begin{tikzcd}
	{\mathrm{Spin}^g(E) \coloneqq}
	&[-3em]
	{\Theta^{-1}(\mathrm{SO}(E))}
	& {\widetilde{\mathrm{GL}^+}(E)} \\
	&
	{\mathrm{SO}(E)}
	& {{\mathrm{GL}^+}(E).}
	\arrow[hook, from=1-2, to=1-3]
	\arrow["\Theta"', from=1-2, to=2-2]
	\arrow["\Theta", from=1-3, to=2-3]
	\arrow[hook, from=2-2, to=2-3]
\end{tikzcd}
\]
In fact, since the inclusion $\mathrm{SO}(n) \hookrightarrow \mathrm{GL}^+_n$ is a homotopy equivalence, every metric spin structure arises in the above fashion.

Let $Z \to B$ be a fiber bundle, whose fiber is a smooth manifold of dimension $n$. Suppose that $Z \to B$ is equipped with a spin structure extending a given spin structure $\mathfrak{s}$ on $M$. A choice of Riemannian metric $g : B \to \mathcal{R}^{\text{fib}}(T^vZ)$ along the fibers of $Z \to B$ gives a positive definite bundle metric on $T^vZ \to Z$, and consequently induces a metric spin structure $\tilde{\mathfrak{s}}^g$ on $T^vZ \to Z$. As above, this gives a Clifford bundle $\slashed{\mathfrak{S}}^g(T^vZ) \to Z$, and composing with the map $Z \to B$, we get a fiber bundle $\slashed{\mathfrak{S}}^g(T^vZ) \to B$, whose fiber over a point $b \in B$ is the $\cl_n$-linear spinor bundle $\slashed{\mathfrak{S}}^g(T(Z_b)) \to Z_b$ associated to the Riemannian spin manifold $(Z_b, g\vert_{Z_b}, \tilde{\mathfrak{s}}^g\vert_{Z^b})$.

\subsection{Atiyah-Singer family index theorem \& the index difference}
The paper \cite{Ebert2016} is especially well suited for the purposes of this paper, and we follow it closely in this section (see also \cite{Botvinnik-Ebert-Randal-Williams}). We will not need the full power of \cite{Ebert2016}, and to not obscure the conceptually clear picture, we only recall the parts which are necessary for this article.

Let $(X, Y)$ be a CW-pair. A $p$-cycle on $(X, Y)$ is a pair $(V, F)$, where $V \to X$ is a~$\cl_p$-Hilbert bundle, and $F$ is a Fredholm family on $V \to X$, in the sense of \cite{Ebert2016}, which is invertible over $Y$. There are straightforward notions of pullbacks, sums, and isomorphisms of $p$-cycles. Two $p$-cycles $(V_0, F_0)$ and $(V_1, F_1)$ are called concordant if there is a $ p$-cycle $(V,F)$ on $(X \times [0,1], Y \times [0,1])$ such that the restriction of $(V, F)$ to $X \times \{i\}$ is isomorphic to $(V_i, F_i)$, for $i \in \{0,1\}$. A $p$-cycle $(V,F)$ on $(X,Y)$ is called acyclic if $F$ is invertible over $X$. The set of concordance classes of $p$-cycles forms an abelian monoid, and the quotient of this monoid by the submonoid of acyclic $p$-cycles forms an abelian group, which we will call $F^p(X,Y)$. Further, $F^p$ defines a functor
\begin{align}
    F^p : \mathbf{CWPair}^{\mathrm{op}} \to \mathbf{Ab}, 
\end{align}
and in \cite{Ebert2016}, it is shown that there is a natural isomorphism of functors
\begin{align}\label{natural-iso}
    \mathrm{Ind}:F^p \Rightarrow \KO^{-p}.
\end{align}

Let $Z$ and $B$ be CW-complexes and let $ p :Z \to B$ be a fiber bundle, whose fiber is a closed smooth manifold $M$ of dimension $n$. Suppose that the fiber bundle $Z \to B$ is equipped with a spin structure, extending some spin structure on $M$. A choice of Riemannian metric $g : B \to \mathcal{R}^{\text{fib}}(T^vZ)$ along the fibers of $Z \to B$ gives, as in Section \ref{metric spin structures}, a bundle $\slashed{\mathfrak{S}}^g(T^vZ) \to B$ of Clifford-linear spinor bundles, where each fiber $\slashed{\mathfrak{S}}^g(T^vZ)_b \to Z_b$ is the $\cl_n$-linear spinor bundle on the manifold $Z_b$ with respect to the induced spin structure and metric. The Hilbert spaces $L^2(\slashed{\mathfrak{S}}^g(T^vZ)_b \to Z_b)$ form a $\cl_n$-Hilbert bundle $L^2(\slashed{\mathfrak{S}}^g(T^vZ)) \to B$, and the family $D^g : b \mapsto \frac{\slashed{D}^g_b}{(1+(\slashed{D}_{b}^{g})^2)^{1/2}}$ of fiberwise Dirac operators defines a Fredholm family on $L^2(\slashed{\mathfrak{S}}^g(T^vZ)) \to B$. Hence the pair $(L^2(\slashed{\mathfrak{S}}^g(T^vZ)), D^g)$ defines an element $[(L^2(\slashed{\mathfrak{S}}^g(T^vZ)), D^g)] \in F^n(B)$, and the corresponding element $\mathrm{Ind}([(L^2(\slashed{\mathfrak{S}}^g(T^vZ)), D^g)]) \in \KO^{-n}(B)$ is called the analytic index of the Fredholm family $D^g$. The spin structure on $Z \to B$ induces a map $p_! : \KO(Z) \to \KO^{-n}(B)$ and the Atiyah-Singer family index theorem states that
\begin{align}\label{Atiyah-Singer}
    \mathrm{Ind}([(L^2(\slashed{\mathfrak{S}}^g(T^vZ)), D^g)]) = p_!(1).
\end{align}
If the Riemannian metric along the fibers, $g$, that we started with is such that $g$ restricts to a section of $\mathcal{R}^{\text{fib}}_{\text{inv}}(T^vZ)\vert_C \to C$ for some sub-complex $C \subset B$, then by the same procedure as above, we get an element in $F^n(B,C)$, which we also denote by $[(L^2(\slashed{\mathfrak{S}}^g(T^vZ)), D^g)]$. It will be clear from context whether we consider it as an element of $F^n(B,C)$ or $F^n(B)$, and if both are used at the same time, then we will distinguish them by pulling one back with the canonical inclusion  $\iota: (B, \emptyset) \to (B, C)$, which corresponds to forgetting that $g$ restricts to a section of $\mathcal{R}^{\text{fib}}_{\text{inv}}(T^vZ)\vert_C \to C$. 

We now define the index difference. Let $g : (D^{\ell+1}, S^{\ell}, *) \to (\mathcal{R}(M), \mathcal{R}_{\mathrm{inv}}(M), g_0)$ represent an element $[g] \in \pi_{\ell+1}(\mathcal{R}(M), \mathcal{R}_{\mathrm{inv}}(M), g_0) \cong \pi_{\ell}(\mathcal{R}_{\mathrm{inv}}(M), g_0)$. The family $g$ defines a Riemannian metric along the fibers of the trivial fiber bundle $\underline{M} := D^{\ell+1} \times M \to D^{\ell+1}$. As before, we obtain a $\cl_n$-Hilbert bundle $L^2(\slashed{\mathfrak{S}}^g(T^v\underline{M})) \to D^{\ell + 1}$ together with a Fredholm family $D^g$, which together form an element $[(L^2(\slashed{\mathfrak{S}}^g(T^v\underline{M})), D^g)] \in F^n(D^{\ell+1}, S^{\ell})$, and consequently defines an element $\mathrm{Ind}[(L^2(\slashed{\mathfrak{S}}^g(T^v\underline{M})), D^g)] \in \KO^{-n}(D^{\ell+1}, S^{\ell}) \cong \KO^{-(n+\ell+1)}(*)$. By homotopy-invariance of the index, this defines a map
\begin{align}
    \mathrm{\inddiff} :\pi_{\ell}(\mathcal{R}_{\mathrm{inv}}(M), g_0) &\to \KO^{-(n+\ell+1)}(*), \\
    [g] &\mapsto \mathrm{Ind}[(L^2(\slashed{\mathfrak{S}}^g(T^v\underline{M})), D^g)] \nonumber
\end{align}
which is called the index difference.

\section{Surface bundles over surfaces with non-zero signature}
In this section, we review known results regarding surface bundles over surfaces with non-zero signature, which will be needed for the proof of the main theorem. Such bundles were first constructed by Kodaira and Atiyah \cite{Kod67, Atiyah69}. Meyer \cite[Satz 3]{Meyer}, showed that for any $\gamma \geq 3$, there is a surface bundle $Z \to \Sigma_h$, with fiber $\Sigma_{\gamma}$, having $\sigma(Z) \neq 0$. The existence of such surface bundles with the extra condition of being spin is then a standard consequence, which we phrase as Lemma \ref{spin-kodiara-lemma} below.

\begin{lemma}\label{spin-kodiara-lemma}
    Let $\Sigma_{\gamma}$ be a closed, oriented surface of genus $\gamma \geq 3$. For any spin structure $\mathfrak{s}$ on $\Sigma_{\gamma}$, there exists a smooth fiber bundle $p : Z \to \Sigma_{h}$ with fiber $\Sigma_{\gamma}$ such that
    \begin{enumerate}
        \item $\sigma(Z) \neq 0$, and \label{signature}
        \item there is a spin structure on $p : Z \to \Sigma_{h}$ which extends $\mathfrak{s}$.
    \end{enumerate}
\end{lemma}
\begin{proof}
    Fix some $\gamma \geq 3$ and let $q : E \to \Sigma_{g}$ be a smooth fiber bundle with fiber $\Sigma_{\gamma}$ such that $\sigma(E) \neq 0$ \cite[Satz 3]{Meyer}. Let $b$ be a base point of $\Sigma_g$ and let $\rho : \pi_1(\Sigma_{g}, b) \to \pi_0(\mathrm{Diff}(\Sigma_{\gamma}))$ denote the monodromy representation of the bundle $q : E \to \Sigma_g$. Fix any spin structure $\mathfrak{s}$ on $\Sigma_{\gamma}$. Since the number of isomorphism classes of spin structures on $\Sigma_{\gamma}$ is finite, the orbit of $\mathfrak{s}$ under the action of the group of isotopy classes of diffeomorphisms (see (\ref{action of pi0})) is a finite set, and therefore the stabilizer $\pi_0(\Diff{\Sigma_{\gamma}, \mathfrak{s}})$ is a finite index subgroup. Hence the preimage $\rho^{-1}(\pi_0(\Diff{\Sigma_{\gamma}, \mathfrak{s}}))$ is a finite index subgroup of $\pi_1(\Sigma_g, b)$. Let $\Sigma_{g'} \to \Sigma_g$ be the finite covering corresponding to the finite index subgroup $\rho^{-1}(\pi_0(\Diff{\Sigma_{\gamma}, \mathfrak{s}})) \subset \pi_1(\Sigma_g, b)$. Lift $q : E \to \Sigma_g$ along this covering and  let $q' : E' \to \Sigma_{g'}$ be the corresponding smooth $\Sigma_{\gamma}$-bundle. Note that $\sigma(E') \neq 0$, and that the monodromy of $q': E' \to \Sigma_{g'}$ has image contained inside $\pi_0(\Diff{\Sigma_{\gamma}, \mathfrak{s}})$. Let $\hat{q}' : \Sigma_{g'} \to \mathrm{B}\Diff{\Sigma_{\gamma}, \mathfrak{s}}$ be a classifying map for $q'$. If the class $(\hat{q}')^*c(\Sigma_{\gamma}, \mathfrak{s})$ is zero, then by Proposition \ref{spin-along-fibers}, there exists a spin structure on $q' : E' \to \Sigma_{g'}$ extending $\mathfrak{s}$. If the class $(\hat{q}')^*c(\Sigma_{\gamma}, \mathfrak{s})$ is not zero, then we take a twofold covering $\Sigma_{g''} \to \Sigma_{g'}$ and lift the bundle $q' : E' \to \Sigma_{g'}$ over this covering, to get the bundle $q'': E'' \to \Sigma_{g''}$, which satisfies $(\hat{q}'')^*c(\Sigma_{\gamma}, \mathfrak{s}) = 0$, where $\hat{q}'' : \Sigma_{g''} \to \mathrm{B}\Diff{\Sigma_{\gamma}, \mathfrak{s}}$ is the classifying map for $q''$. Note again that $\sigma(E'') \neq 0$. In either case, we get a bundle $p : Z \to \Sigma_{h}$ which satisfies the desired conditions.
\end{proof}

The reader who is only interested in the first part of Theorem \ref{main theorem}, namely the non-triviality of the index difference, is ready to read the proof of this in Section \ref{last_section} below. To prove that the index difference in Theorem \ref{main theorem} is surjective if the genus is at least $5$, we need that the signature of the corresponding spin surface bundle is equal to $16$. This is the content of Lemma \ref{spin-signature-16} below, which follows from the work of Putman \cite{Putman}. However, we will instead combine the work of Randal-Williams \cite{RANDALWILLIAMS, RandalWilliams-stability} with the work of Sierra \cite{Sierra} to deduce the aforementioned lemma. 

For the convenience of the reader, we recall some basic facts regarding the signature of surface bundles in the specific cases relevant for this exposition. We also deduce the content of Lemma \ref{spin-signature-16} without assuming much prior knowledge of surface bundles.

Let $p: Z \to B$ be an oriented surface bundle over a surface, and let $T^vZ \to Z$ denote the vertical tangent bundle. This is an oriented real $2$-plane bundle, and can hence be viewed as a complex line bundle. Thus, the first Chern class and the Euler class of $T^vZ \to Z$ agree, i.e. we have that
\begin{align*}
    c_1(T^vZ) = e(T^vZ) \in H^2(Z; \mathbb{Z}).
\end{align*}
Atiyah \cite{Atiyah69} shows, using the Hirzebruch signature formula, that
\begin{align*}
    \sigma(Z) = \frac{\esc{c_1(T^vZ)^2, [Z]}}{3}.
\end{align*}
Let $p_! : H^4(Z; \mathbb{Z}) \to H^2(B; \mathbb{Z})$ be the Gysin map in integral cohomology, and put
\begin{align*}
    \kappa_1(p) := p_!(e(T^vZ)^2) = p_!(c_1(T^vZ)^2) \in H^2(B; \mathbb{Z}).
\end{align*}
Hence, we see that the signature can also be computed as
\begin{align}\label{signature-kappa}
    \sigma(Z) = \frac{\esc{\kappa_1(p), [B]}}{3}.
\end{align}
Let $\mathcal{U}:\mathrm{E}_{\gamma} \to \mathrm{BDiff}(\Sigma_{\gamma})$ be the universal oriented $\Sigma_{\gamma}$-bundle and let $T^v\mathrm{E}_{\gamma} \to \mathrm{E}_{\gamma}$ be the corresponding vertical tangent bundle (see Section \ref{spin structures on fiber bundles}). Let $e(T^v\mathrm{E}_{\gamma}) \in H^2(\mathrm{E}_{\gamma};\mathbb{Z})$ be the Euler class of this $2$-plane bundle and put
\begin{align*}
    \kappa_1 := \mathcal{U}_!(e(T^v\mathrm{E}_{\gamma})^2) \in H^2(\mathrm{BDiff}(\Sigma_{\gamma}); \mathbb{Z}).
\end{align*}
Let $f :B \to \mathrm{BDiff}(\Sigma_{\gamma})$ be the classifying map for an oriented $\Sigma_{\gamma}$-bundle $p : Z \to B$. By construction, we have that $f^*\kappa_1= \kappa_1(p)$, and so by \eqref{signature-kappa}, we have that
\begin{align}\label{signature-universal-kappa}
    \sigma(Z) = \frac{\esc{f^*\kappa_1, [B]}}{3} = \frac{\esc{\kappa_1, f_*[B]}}{3}.
\end{align}
Since every class in the second homology of a space can be represented by a map from a closed oriented surface, it is enough to study the pairing
\begin{align}
    \sign_{\gamma}: H_2(\mathrm{BDiff}(\Sigma_{\gamma}); \mathbb{Z}) &\to \mathbb{Z} \\ \nonumber
    x& \mapsto \frac{\esc{\kappa_1, x}}{3},
\end{align}
in order to study the signature of surface bundles over surfaces.

We will also need to consider surfaces with one boundary component. Let $\Sigma_{\gamma, 1}$ be an oriented surface of genus $\gamma$ with one boundary component, and let
\begin{align*}
    \mathrm{Diff}_{\partial}({\Sigma_{\gamma, 1}})
\end{align*}
be the group of orientation-preserving diffeomorphisms that pointwise fix a neighborhood of the boundary. Let $\mathfrak{s}$ be any spin structure on $\Sigma_{\gamma, 1}$ and let
\begin{align*}
    \mathrm{Diff}_{\partial}({\Sigma_{\gamma, 1}}, \mathfrak{s}) \subset \mathrm{Diff}_{\partial}({\Sigma_{\gamma, 1}})
\end{align*}
be the subgroup consisting of those diffeomorphisms $f$ which preserve the spin structure $\mathfrak{s}$, meaning that $f^*{\mathfrak{s}} \cong \mathfrak{s}$. Fix a trivialization of $\mathfrak{s}$ over the boundary of $\Sigma_{\gamma, 1}$ and note that the two possible choices of isomorphism $f^*{\mathfrak{s}} \cong \mathfrak{s}$ act as the identity, respectively minus the identity on this boundary trivialization. Write $+_f$ for the isomorphism which acts as the identity on this boundary trivialization. Let $\mathrm{SDiff}_{\partial}({\Sigma_{\gamma, 1}}, \mathfrak{s})$ be the group of spin diffeomorphisms, defined analogously as in Section \ref{spin structures on fiber bundles}. With this, we can define the lift:
\begin{align}\label{spin-lift}
    \mathrm{Diff}_{\partial}({\Sigma_{\gamma, 1}}, \mathfrak{s}) & \to \mathrm{SDiff}_{\partial}({\Sigma_{\gamma, 1}}, \mathfrak{s}) \\
    f &\mapsto (f, +_f).
\end{align}
This means that when considering diffeomorphism groups which act as the identity on the boundary, the $\mathbb{Z}/2\mathbb{Z}$-ambiguity concerning the choice of spin isomorphism discussed around \eqref{group-extension} and \eqref{fibration of group extension} in Section \ref{spin structures on fiber bundles} can be avoided; this is the main reason why we consider manifolds with boundary in this context. Hence, any map $f : B \to \mathrm{BDiff}_{\partial}({\Sigma_{\gamma, 1}}, \mathfrak{s})$ gives rise to a $\Sigma_{\gamma, 1}$-bundle $E_{\partial} \to B$ which has a spin structure that extends the given spin structure $\mathfrak{s}$ on $\Sigma_{\gamma, 1}$ (see Section \ref{spin structures on fiber bundles}). Since the spin structure $\mathfrak{s}$ induces a bounding spin structure on the boundary, the spin structure $\mathfrak{s}$ induces a spin structure on $\Sigma_{\gamma}$, which we continue to denote by $\mathfrak{s}$. Since any $f \in \mathrm{Diff}_{\partial}({\Sigma_{\gamma, 1}}, \mathfrak{s})$ preserves a neighborhood of the boundary, we can cap off each fiber of $E_{\partial} \to B$ with a disk and obtain a spin $\Sigma_{\gamma}$-bundle $E \to B$ with $\sigma(E) = \sigma(E_{\partial})$. Let
\begin{align}\label{capping}
    \mathrm{BDiff}_{\partial}({\Sigma_{\gamma, 1}}, \mathfrak{s}) \to \mathrm{BSDiff}(\Sigma_{\gamma}, \mathfrak{s})
\end{align}
be the map induced by the map \eqref{spin-lift}, together with the map induced by capping off the boundary component of $({\Sigma_{\gamma, 1}}, \mathfrak{s})$ with a disk and extending any diffeomorphism $f \in \mathrm{Diff}_{\partial}({\Sigma_{\gamma, 1}}, \mathfrak{s})$ by the identity. We write
\begin{align*}
    \mathrm{cap}_{\gamma}: H_2(\mathrm{BDiff}_{\partial}({\Sigma_{\gamma, 1}, \mathfrak{s}});\mathbb{Z}) \to H_2(\mathrm{BSDiff}(\Sigma_{\gamma}, \mathfrak{s});\mathbb{Z})
\end{align*}
for the induced map in homology. This discussion shows that in order to find a spin $(\Sigma_{\gamma}, \mathfrak{s})$-bundle over a surface with signature equal to $16$, it suffices to find an element 
\begin{align*}
    x \in H_2(\mathrm{BDiff}_{\partial}(\Sigma_{\gamma, 1}, \mathfrak{s});\mathbb{Z})
\end{align*}
such that
\begin{align*}
    \sign_{\gamma}(\mathrm{cap}_{\gamma}(x)) = 16.
\end{align*}
Fix a decomposition $\Sigma_{\gamma +1, 1} = \Sigma_{\gamma, 1} \cup_{S^1} \Sigma_{1, 2}$, where $\Sigma_{1,2}$ is a surface of genus $1$ with $2$ boundary components, let $\mathfrak{s}'$ be a spin structure on $\Sigma_{\gamma +1, 1}$, and write $\mathfrak{s}$ for the induced spin structure on $\Sigma_{\gamma, 1}$. Extending any diffeomorphism $f \in \mathrm{Diff}_{\partial}({\Sigma_{\gamma, 1}}, \mathfrak{s})$ via the identity to a diffeomorphism of $\Sigma_{\gamma +1, 1}$ gives a stabilization map
\begin{align*}
    \mathrm{BDiff}_{\partial}({\Sigma_{\gamma, 1}}, \mathfrak{s}) \to \mathrm{BDiff}_{\partial}({\Sigma_{\gamma+1, 1}}, \mathfrak{s}')
\end{align*}
and we write 
\begin{align*}
    \mathrm{stab}_{\gamma} : H_2(\mathrm{BDiff}_{\partial}({\Sigma_{\gamma, 1}}, \mathfrak{s});\mathbb Z) \to H_2(\mathrm{BDiff}_{\partial}({\Sigma_{\gamma+1, 1}}, \mathfrak{s}');\mathbb Z)
\end{align*}
for the induced map in homology. Sometimes we will write $\mathfrak{s}$ both for the spin structure on $\Sigma_{\gamma +1, 1}$ and $\Sigma_{\gamma, 1}$. This notation will implicitly mean that the spin structure on $\Sigma_{\gamma +1, 1}$ is an extension of the spin structure on $\Sigma_{\gamma, 1}$ in such a way that does not change the Arf invariant. Note that by construction, the signature is unaffected by stabilizing, i.e. we have that
\begin{align}\label{sign-stab}
    \sign_{\gamma+1}(\mathrm{cap}_{\gamma+1}(\mathrm{stab}_{\gamma}(x))) = \sign_{\gamma}(\mathrm{cap}_{\gamma}(x)), 
\end{align}
for any $x \in H_2(\mathrm{BDiff}_{\partial}({\Sigma_{\gamma, 1}}, \mathfrak{s});\mathbb Z)$.

One reason surface bundles are somewhat more manageable than general manifold bundles is the following simplification. We write  
\begin{align*}
    \Gamma_{\gamma, 1}[\mathfrak{s}] &:= \pi_0(\mathrm{Diff}_{\partial}({\Sigma_{\gamma, 1}}, \mathfrak{s})), \text{ and} \\
    \Gamma_{\gamma}[\mathfrak{s}] &:= \pi_0(\mathrm{SDiff}({\Sigma_{\gamma}}, \mathfrak{s}))
\end{align*}
to denote the corresponding groups of isotopy classes, which are called the spin mapping class groups. Let $\mathrm{Diff}_{\partial}^0({\Sigma_{\gamma, 1}}, \mathfrak{s})$ be the identity component of $\mathrm{Diff}_{\partial}({\Sigma_{\gamma,1}}, \mathfrak{s})$. By the work of Earle and Eells \cite{Earle-Eells-69, Earle-Schatz-70}, we have that $\mathrm{Diff}_{\partial}^0({\Sigma_{\gamma, 1}}, \mathfrak{s})$ is contractible, and hence the short exact sequence
\begin{align*}
    1 \to \mathrm{Diff}_{\partial}^0({\Sigma_{\gamma, 1}}, \mathfrak{s}) \to \mathrm{Diff}_{\partial}({\Sigma_{\gamma, 1}}, \mathfrak{s}) \to \Gamma_{\gamma, 1}[\mathfrak{s}] \to 1
\end{align*}
induces a homotopy equivalence $\mathrm{BDiff}_{\partial}({\Sigma_{\gamma, 1}}, \mathfrak{s}) \simeq \mathrm{B}\Gamma_{\gamma, 1}[\mathfrak{s}]$. The analogous statement is true for $\mathrm{SDiff}^0({\Sigma_{\gamma}}, \mathfrak{s})$, for $\gamma \geq 2$, which follows from \cite{Earle-Eells-69}. Thus we have a homotopy equivalence $\mathrm{BSDiff}({\Sigma_{\gamma}}, \mathfrak{s}) \simeq \mathrm{B}\Gamma_{\gamma}[\mathfrak{s}]$ also in this case. 

Note that both the maps $\mathrm{cap}_{\gamma}$ and $\mathrm{stab}_{\gamma}$ can be defined directly on the mapping class group side and that these maps commute with the above homotopy equivalences. We transport the $\kappa_1$ class via the homotopy equivalence and write $\kappa_1$ also for this class, and note that the corresponding signature homomorphism commutes with the previously mentioned maps. Hence, by \eqref{sign-stab}, we have that the diagram

\[\begin{tikzcd}[
  column sep=small,
  row sep=normal,
  cells={nodes={font=\small}}
]
	& {\mathbb{Z}} && \\
	& {H_2(\mathrm{BSDiff}(\Sigma_{\gamma}, \mathfrak{s}); \mathbb{Z})} & {H_2(\mathrm{B}\Gamma_{\gamma}[\mathfrak{s}]; \mathbb{Z})} \\
	& {H_2(\mathrm{BDiff}(\Sigma_{\gamma, 1}, \mathfrak{s}); \mathbb{Z})} & {H_2(\mathrm{B}\Gamma_{\gamma, 1}[\mathfrak{s}]; \mathbb{Z})} \\
	{H_2(\mathrm{BSDiff}(\Sigma_{\gamma-1}, \mathfrak{s}); \mathbb{Z})} & {H_2(\mathrm{BDiff}(\Sigma_{\gamma-1, 1}, \mathfrak{s}); \mathbb{Z})} & {H_2(\mathrm{B}\Gamma_{\gamma-1, 1}[\mathfrak{s}]; \mathbb{Z})} & {H_2(\mathrm{B}\Gamma_{\gamma-1}[\mathfrak{s}]; \mathbb{Z})} \\
	& {} & {}
	\arrow["{\sign_{\gamma}}", from=2-2, to=1-2]
	\arrow["\cong", from=2-2, to=2-3]
	\arrow["{\sign_{\gamma}}"', from=2-3, to=1-2]
	\arrow["{\mathrm{cap}_{\gamma}}", from=3-2, to=2-2]
	\arrow["\cong", from=3-2, to=3-3]
	\arrow["{\mathrm{cap}_{\gamma}}"', from=3-3, to=2-3]
	\arrow["{\sign_{\gamma-1}}", curve={height=-24pt}, from=4-1, to=1-2]
	\arrow["{\mathrm{stab}_{\gamma-1}}", from=4-2, to=3-2]
	\arrow["{\mathrm{cap}_{\gamma-1}}"', from=4-2, to=4-1]
	\arrow["\cong", from=4-2, to=4-3]
	\arrow["{\mathrm{stab}_{\gamma-1}}"', from=4-3, to=3-3]
	\arrow["{\mathrm{cap}_{\gamma-1}}", from=4-3, to=4-4]
	\arrow["{\sign_{\gamma-1}}"', curve={height=70pt}, from=4-4, to=1-2]
	\arrow["{\mathrm{stab}_{\gamma-2}}", dotted, from=5-2, to=4-2]
	\arrow["{\mathrm{stab}_{\gamma-2}}"', dotted, from=5-3, to=4-3]
\end{tikzcd}\]

commutes. The stability of the map $\mathrm{stab}_{\gamma}$ has been studied by many and the latest result in this direction is given by Sierra \cite[Theorem A]{Sierra}, which states that it is surjective for $\gamma \geq 5$ and an isomorphism for $\gamma \geq 7$. Similarly, Randal-Williams \cite[Theorem 2.14]{RandalWilliams-stability} has shown that $\mathrm{cap}_{\gamma}$ is surjective for $\gamma \geq 4$ and an isomorphism for $\gamma \geq 7$.

Randal-Williams \cite[Example 1.10]{RANDALWILLIAMS} has shown for any $\gamma \geq 9$ that
\begin{align}
    \langle \lambda, \mu \mid 4(\lambda + 4\mu)\rangle
\end{align}
is a presentation of $H^2(\mathrm{B}\Gamma_{\gamma}[\mathfrak{s}];\mathbb{Z})$, where $12\lambda = \kappa_1$. This means that $[\mu]$ is a generator of the torsion-free part and that $[\lambda] = -4[\mu]$ in the torsion-free part. Hence, for any $\gamma \geq 9$, we can find an $x \in H_2(\mathrm{B}\Gamma_{\gamma}[\mathfrak{s}];\mathbb{Z})$ such that $\esc{\mu, x}=-1$, and note that
\begin{align*}
    \sign_{\gamma}(x) = \frac{\esc{\kappa_1, x}}{3} = 4\esc{\lambda, x} = -16\esc{\mu, x} = 16.
\end{align*}
Since $\mathrm{stab}_{\gamma}$ and $\mathrm{cap}_{\gamma}$ are surjective for $\gamma \geq 5$, we can use \eqref{sign-stab} to conclude that for any $\gamma \geq 5$, we can find an element $x \in H_2(\mathrm{B}\Gamma_{\gamma, 1}[\mathfrak{s}];\mathbb{Z})$ such that $\sign_{\gamma}(\mathrm{cap}_{\gamma}(x))= 16$.

Hence, we have deduced the following lemma.

\begin{lemma}\label{spin-signature-16}
    Let $\Sigma_{\gamma}$ be a closed, oriented surface of genus $\gamma \geq 5$. For any spin structure $\mathfrak{s}$ on $\Sigma_{\gamma}$, there exists a smooth fiber bundle $p : Z \to \Sigma_{h}$ with fiber $\Sigma_{\gamma}$ such that
    \begin{enumerate}
        \item $\sigma(Z)= 16$, and \label{signature=16}
        \item there is a spin structure on $p : Z \to \Sigma_{h}$ which extends $\mathfrak{s}$.
    \end{enumerate}
\end{lemma}

\begin{remark}
    The reference to a specific spin structure $\mathfrak{s}$ is not directly relevant for discussing diffeomorphism groups in this case; we only chose it for brevity and for our specific application. In fact, if $\mathfrak{s}_0$ and $\mathfrak{s}_1$ are spin structures on $\Sigma_{\gamma, 1}$ with the same Arf invariant, then $\mathrm{Diff}_{\partial}({\Sigma_{\gamma, 1}}, \mathfrak{s}_0)$ and $\mathrm{Diff}_{\partial}({\Sigma_{\gamma, 1}}, \mathfrak{s}_1)$ are conjugate subgroups of $\mathrm{Diff}_{\partial}({{\Sigma_{\gamma, 1}}})$ \cite[Corollary 2]{Johnson}. It is also worth pointing out that many of the standard facts referred to above do not use spin and could have been stated separately. The reader should also be aware that the references used in this section all use group (co)-homology notation and write $H^q(\Gamma_{\gamma};\mathbb{Z})$ to mean $H^q(\mathrm{B}\Gamma_{\gamma};\mathbb{Z})$. We chose to avoid such notation for extra clarity to the reader who is unfamiliar with this subject.
\end{remark}

\section{Proofs of Theorem \ref{main theorem} and Theorem \ref{second main theorem}}\label{last_section}

We begin by a sketch of the proof of Theorem \ref{main theorem}. First, we take a spin $(\Sigma_{\gamma}, \mathfrak{s})$-bundle $Z \to \Sigma_h$ over a surface $\Sigma_h$, with $\sigma(Z) \neq 0$ which exists by Lemma \ref{spin-kodiara-lemma}. Then one uses that $\mathcal{R}_{\mathrm{inv}}(\Sigma_{\gamma})$ is connected by Theorem \eqref{Ammann-Dahl} to get a Riemannian metric along the fibers of $Z \to \Sigma_h$, which is invertible on the $1$-skeleton. This means that the fiberwise Dirac operator gives an index in $\KO^{-2}(D^2, S^1) \cong \KO^{-4}(*)$, and using that $\sigma(Z) \neq 0$, it follows by standard arguments that this index is non-zero. This implies that the index difference is non-zero, and to show surjectivity if $\gamma \geq 5$ we use that, in these cases, we can chose $Z \to \Sigma_h$ above with $\sigma(Z) = 16$.

\begin{proof}[Proof of Theorem \ref{main theorem}]
 Let $\Sigma_{\gamma}$ be a closed, oriented surface of genus $\gamma \geq 3$, equipped with any bounding spin structure $\mathfrak{s}$ and let $p: Z \to \Sigma_{h}$ be a smooth fiber bundle with fiber $\Sigma_{\gamma}$ supplied by Lemma \ref{spin-kodiara-lemma} and let $P_Z \to \Sigma_h$ denote the associated $\SDiff{\Sigma_{\gamma}, \mathfrak{s}}$-principal bundle. Fix a CW-structure on $\Sigma_h$ with one $0$-cell $b$ and one $2$-cell, and let $\Sigma_h^1$ be the $1$-skeleton of this CW-structure. Since $\mathcal{R}_{\mathrm{inv}}(\Sigma_{\gamma})$ is connected by Theorem \ref{Ammann-Dahl}, we have as in section \ref{metrics-along-fibers} that there exists a continuous section $\hat{g} : (\Sigma_h, \Sigma_{h}^{1}, b) \to (\mathcal{R}^{\mathrm{fib}}(T^vZ), \mathcal{R}^{\mathrm{fib}}_{\mathrm{inv}}(T^vZ), [f,g_0])$.

 This defines an element $[(L^2(\slashed{\mathfrak{S}}^{\hat{g}}(T^vZ)), {D}^{\hat{g}})] \in F^2(\Sigma_h, \Sigma_h^1)$. Via the canonical inclusion $\iota : (\Sigma_h, \emptyset) \to (\Sigma_h, \Sigma_h^1)$ we also get an element $\iota^*[(L^2(\slashed{\mathfrak{S}}^{\hat{g}}(T^vZ)), {D}^{\hat{g}})] \in F^2(\Sigma_h)$. By the Atiyah-Singer family index theorem (\ref{Atiyah-Singer}), we have that $\mathrm{Ind}(\iota^*[(L^2(\slashed{\mathfrak{S}}^{\hat{g}}(T^vZ)), {D}^{\hat{g}})])= p_!(1) \in {\KO}^{-2}(\Sigma_{h})$. Since $TZ \cong T^vZ \oplus p^*(T\Sigma_{h})$, we have that the spin structure on $T^vZ$, together with a spin structure on $\Sigma_{h}$, induces a spin structure on $Z$. Equip $\Sigma_h$ with any spin structure and let $*^Z, *^{\Sigma_{h}}$ denote the maps from $Z$ and $\Sigma_{h}$ to a point. By functoriality of the Gysin map we have that the diagram
\[\begin{tikzcd}
	{\KO(Z)} & {\KO^{-2}(\Sigma_{h})} \\
	{\KO^{-4}(*)}
	\arrow["{p_!}", from=1-1, to=1-2]
	\arrow["{*^Z_!}"', from=1-1, to=2-1]
	\arrow["{*^{\Sigma_{h}}_!}", from=1-2, to=2-1]
\end{tikzcd}\]
commutes. By the Atiyah-Singer index theorem, we have that under the standard isomorphism \newline $\KO^{-4}(*) \overset{\cong}{\to} \mathbb{Z}$, the class $*^Z_!(1)$ maps to $\frac{\hat{A}(Z)}{2} = -\frac{\sigma(Z)}{16}$, which is non-zero by Lemma \ref{spin-kodiara-lemma}. Hence, we conclude that the element $[(L^2(\slashed{\mathfrak{S}}^{\hat{g}}(T^vZ)), {D}^{\hat{g}})] \in F^2(\Sigma_h, \Sigma_h^1)$ is non-zero.
 
Let $\Phi : (D^2, S^1, *) \to (\Sigma_h, \Sigma_h^1, b)$ be the characteristic map of the $2$-cell in $\Sigma_h$. Since $\Phi$ induces a homeomorphism $\Phi : D^2/S^1 \to \Sigma_h/\Sigma_h^1$, we get a group isomorphism 
\begin{align*}
    \Phi^* : F^2(\Sigma_h, \Sigma_h^1) \overset{\cong}{\to} F^2(D^2, S^1),
\end{align*}
sending the non-zero element $[(L^2(\slashed{\mathfrak{S}}^{\hat{g}}(T^vZ)), {D}^{\hat{g}})] \in F^2(\Sigma_h, \Sigma_h^1)$ to the non-zero element $[(\Phi^*L^2(\slashed{\mathfrak{S}}^{\hat{g}}(T^vZ)), \Phi^*{D}^{\hat{g}})]\in F^2(D^2, S^1)$. Now, let 
\begin{align*}
    F: \Phi^*P_Z \to D^2 \times \SDiff{\Sigma_{\gamma}, \mathfrak{s}}
\end{align*}
be a trivialization with $F((*, f)) = (*,\id)$ (see the definition of the section $\hat{g}$ above). This induces trivializations of the associated bundles
\begin{align*}
    &F : \Phi^*Z \to D^2 \times \Sigma_{\gamma} = \underline{\Sigma}_{\gamma}\\\
    &F :\Phi^*(\mathcal{R}^{\mathrm{fib}}(T^vZ)) \to D^2 \times \mathcal{R}(\Sigma_{\gamma}), 
\end{align*}
which we also denote by $F$. Let $g : (D^2, S^1, *) \to (\mathcal{R}(\Sigma_{\gamma}), \mathcal{R}_{\mathrm{inv}}(\Sigma_{\gamma}), g_0)$ be the family of Riemannian metrics given by $ \mathrm{pr}_2 \circ F\circ \Phi^*\hat{g}$. This family defines an element $[(L^2(\slashed{\mathfrak{S}}^{g}(\underline{\Sigma}_{\gamma})), {D}^{g})] \in F^2(D^2, S^1)$. Since the trivialization $F$ is a trivialization of the corresponding $\SDiff{\Sigma_{\gamma}, \mathfrak{s}}$-principal bundle $\Phi^*P_Z$, we get an isomorphism $(L^2(\slashed{\mathfrak{S}}^{g}(\underline{\Sigma}_{\gamma})), {D}^{g}) \cong (\Phi^*L^2(\slashed{\mathfrak{S}}^{\hat{g}}(T^vZ)), \Phi^*{D}^{\hat{g}})$ of $2$-cycles, which means that they define the same element in $F^2(D^2, S^1)$. Hence, we have that $[(L^2(\slashed{\mathfrak{S}}^{g}(\underline{\Sigma}_{\gamma})), {D}^{g})] \neq 0 \in F^2(D^2, S^1)$, which means that $\inddiff([g]) \neq 0 \in \KO^{-4}(*)$, which proves the first part of the theorem. 

To prove the last part of the theorem, note that the above argument may be summarized as follows. We constructed an element $[g] \in \pi_1(\mathcal{R}_{\mathrm{inv}}(\Sigma_{\gamma}), g_0)$, such that if we follow the diagram
\[\begin{tikzcd}
	{\pi_1(\mathcal{R}_{\mathrm{inv}}(\Sigma_{\gamma}), g_0)} & {\KO^{-4}(*)} \\
	& {\KO^{-2}(D^2,S^1)} \\
	& {F^2(D^2,S^1)} \\
	& {F^2(\Sigma_h,\Sigma_h^1)} \\
	& {F^2(\Sigma_h)} \\
	& {\KO^{-2}(\Sigma_h)} \\
	& {\KO^{-4}(*)}
	\arrow["\inddiff", from=1-1, to=1-2]
	\arrow["\cong", from=2-2, to=1-2]
    \arrow["{\mathrm{Ind}}"', "\cong", from=3-2, to=2-2]
	\arrow["{\Phi^*}"', "\cong", from=4-2, to=3-2]
	\arrow["{\iota^*}", from=4-2, to=5-2]
	\arrow["{\mathrm{Ind}}", from=5-2, to=6-2]
	\arrow["{*^{\Sigma_h}_!}", from=6-2, to=7-2]
\end{tikzcd}\]
to the bottom, we land at the element $*^Z_!(1)$, and since this element was non-zero we concluded that the index difference is non-trivial. Now, if the spin $(\Sigma_{\gamma}, \mathfrak{s})$-bundle $p : Z \to \Sigma_h$ that we started with has $\sigma(Z) = 16$, then the element $*^Z_!(1)$ is a generator of $\KO^{-4}(*)$, and therefore the element $\inddiff([g])$ must also be a generator of $\KO^{-4}(*)$. By Lemma \ref{spin-signature-16}, such a spin $(\Sigma_{\gamma}, \mathfrak{s})$-bundle $p: Z \to \Sigma_h$ exists if $\gamma \geq 5$.
\end{proof}

The proof of Theorem \ref{second main theorem} is similar to the proof of Theorem \ref{main theorem}, but we spell it out for completeness.
\begin{proof}[Proof of Theorem \ref{second main theorem}]
    Let $\Sigma := \Sigma_{\gamma^1} \times \Sigma_{\gamma^2}$ be a product of closed, oriented surfaces of genera $\gamma^1, \gamma^2 \geq~3$, equipped with any spin structure. Let $Z_{1} \to \Sigma_{h^{1}}$, and $Z_{2} \to \Sigma_{h^{2}}$ be bundles supplied by Lemma \ref{spin-kodiara-lemma}, with fiber $\Sigma_{\gamma^i}$, extending the given spin structure on $\Sigma_{\gamma^i}$. Let $p : Z := Z_1\times Z_2 \to \Sigma_{h^1} \times \Sigma_{h^2} =: B$ be the product bundle. Since the $\KO^{-4}(*)$-valued index of $\Sigma$ vanishes (i.e. the $\hat{A}$-genus of $\Sigma$ vanishes), we have by \cite[Theorem 1.3]{Maier} that $\mathcal{R}_{\mathrm{inv}}(\Sigma) \neq \emptyset$. By Theorem \ref{Ammann-Dahl}, we also have that $\pi_0(\mathcal{R}_{\mathrm{inv}}(\Sigma)) = \{*\}$. Suppose for a contradiction that all the groups $\pi_1(\mathcal{R}_{\mathrm{inv}}(\Sigma)), \pi_2(\mathcal{R}_{\mathrm{inv}}(\Sigma)), \pi_3(\mathcal{R}_{\mathrm{inv}}(\Sigma))$ are trivial. Then we have as in Section \ref{metrics-along-fibers} that there exists a section $g : B \to \mathcal{R}^{\mathrm{fib}}_{\mathrm{inv}}(T^vZ)$ and consequently, we have by the Atiyah-Singer family index theorem (\ref{Atiyah-Singer}), that the family index $p_!(1) \in \KO^{-4}(B)$ vanishes. Since $TZ \cong T^vZ \oplus p^*(TB)$, we have that the spin structure on $T^vZ$, together with a spin structure on $B$, induces a spin structure on $Z$. Equip $B$ with any spin structure and let $*^Z, *^{B}$ denote the corresponding maps to a point. By functoriality of the Gysin map we have that the diagram
\[\begin{tikzcd}
	{\KO(Z)} & {\KO^{-4}(B)} \\
	{\KO^{-8}(*)}
	\arrow["{p_!}", from=1-1, to=1-2]
	\arrow["{*^Z_!}"', from=1-1, to=2-1]
	\arrow["{*^{B}_!}", from=1-2, to=2-1]
\end{tikzcd}\]
commutes. By the Atiyah-Singer index theorem, we have that under Bott periodicity \newline $\KO^{-8}(*)~\cong~\mathbb{Z}$, the class $*^Z_!(1)$ maps to $\hat{A}(Z) = \hat{A}(Z_1)\hat{A}(Z_2) = \frac{\sigma(Z_1)\sigma(Z_2)}{64}$ which is non-zero by Lemma \ref{spin-kodiara-lemma}. Hence the family index $p_!(1) \in \KO^{-4}(B)$ is non-zero and we have established the contradiction, which finishes the proof.
\end{proof}

\begin{proof}[How to deduce the conclusion of Remark \ref{Ammann-Dahl conjecture}]
    Let $Z_{1} \to \Sigma_{h^{1}}$, and $Z_{2} \to \Sigma_{h^{2}}$ be bundles supplied by Lemma \ref{spin-kodiara-lemma}, with fiber $\Sigma_{\gamma^i}$, extending the given spin structure on $\Sigma_{\gamma^i}$. Let $p : Z := Z_1\times Z_2 \to \Sigma_{h^1} \times \Sigma_{h^2} =: B$ be the product bundle. Equip $B$ with a CW-structure with one $0$-cell $b$ and one $4$-cell, and let $B^{3}$ denote the $3$-skeleton of $B$.
     
     Now if $\mathcal{R}_{\mathrm{inv}}(\Sigma_{\gamma^1} \times \Sigma_{\gamma^2})$ is $2$-connected, then as in the proof of Theorem \ref{main theorem} there exists a section 
     \begin{align*}
         g : (B, B^{3}, b) \to (\mathcal{R}^{\mathrm{fib}}(T^vZ), \mathcal{R}^{\mathrm{fib}}_{\mathrm{inv}}(T^vZ), [f,g_0]),
     \end{align*}
     and the rest of the proof proceeds in analogy with the proof of Theorem \ref{main theorem}.
\end{proof}

\newpage

\bibliographystyle{abbrv}
\bibliography{references}


\end{document}